\documentclass[12pt]{article}

\usepackage{amssymb,amsmath,amsthm}
\usepackage[margin=1in]{geometry}

\newtheorem*{cor}{Corollary}
\newtheorem{lemma}{Lemma}[section]
\newtheorem{thm}{Theorem}[section]
\newtheorem{q}{Problem}
\theoremstyle{definition}
\newtheorem*{defn}{Definition}

\newcommand{\lcm}[0]{\textup{lcm}}
\newcommand{\slope}[1]{\textup{slope}\left( #1 \right)}
\newcommand{\modslope}[2]{\textup{slope}_{#1}\left( #2 \right)}

\begin{document}

\title{On Additive Complexity of Infinite Words}
\author{Graham Banero \\ Department of Mathematics \\ Simon Fraser University \\ 8888 University Dr. \\ Burnaby, BC V5A 1S6, Canada \\ gbanero@sfu.ca}
\date{}
\maketitle

\begin{abstract}
We consider questions related to the structure of infinite words (over an integer alphabet) with bounded additive complexity, i.e., words with the property that the number of distinct sums exhibited by factors of the same length is bounded by a constant that depends only on the word. We describe how bounded additive complexity impacts the slope of the word and
how a non-erasing morphism may affect the boundedness of a given word's additive complexity. We prove the existence of recurrent words with constant additive complexity equal to any given odd positive integer. In the last section, we discuss a generalization of additive complexity. Our results suggest that words with bounded additive complexity may be viewed as a generalization of balanced words.  
\end{abstract}

\section{Introduction} \label{sec: 1}

In this note we consider infinite words $\omega = x_1x_2x_3 \cdots$ with $x_i \in S$, where $S$ is a set of integers. If $B = y_1y_2 \cdots y_n$ (here $B$ has \emph{length} $n$) is a \emph{factor} of $\omega$, i.e., $\omega = x_1x_2 \cdots x_m y_1y_2 \cdots y_n z_1z_2 \cdots$, then the \emph{sum} of $B$ is defined by $\sum B = y_1 + y_2 + \cdots + y_n$. We are primarily concerned with words $\omega$ for which there is some $M \in \mathbb{N}$ such that for all $n \geq 1$,
\[ \left| \left\{ \sum B: B \textup{ is a factor of $\omega$ with length } n\right\} \right| \leq M. \]
If such an $M$ exists, $\omega$ is said to have \emph{bounded additive complexity}. We will discuss questions such as the following: If $\omega$ is an infinite word with bounded additive complexity, what else can we say about $\omega$? What conditions are equivalent to bounded additive complexity? What closure properties does the set of words with bounded additive complexity possess?

\subsection{Terminology and Motivation} \label{sub: 1.1}

For the remainder of this note, $S,T$ will always denote sets of integers called \emph{alphabets}, and unless explicitly stated otherwise we will assume that $S$ and $T$ are finite. An element $B = x_1x_2\cdots x_n$ of the free monoid $S^*$ generated by $S$ is called a \emph{finite word} over $S$ with \emph{length} $|B| = n$. We define the \emph{sum} of $B$ as 
\[ \sum B = x_1 + x_2 + \cdots + x_n, \]
and for $s \in S$, we let
\[ |B|_s = \left| \left\{ x_i : 1 \leq i \leq n, \: x_i = s \right\} \right|. \]
If there are words $B',B''$ such that $A = B'BB'' \in S^*$, then $B$ is said to be a \emph{factor} of $A$.
The identity element of $S^*$, called the \emph{empty word}, has length $0$ and sum $0$, and is a factor of every word.

An \emph{infinite word} over $S$ is a sequence $\omega: \mathbb{N} \mapsto S$ and is usually written $\omega = x_1x_2x_3 \cdots$, where $\omega(i) = x_i \in S$ for all $i \in \mathbb{N}$. The set of infinite words over $S$ is denoted by $S^\mathbb{N}$, and we let $S^\infty = S^* \cup S^\mathbb{N}$. For $n \geq m \in \mathbb{N}$, we let $\omega[m,n] = x_mx_{m+1}\cdots x_n$, and the word $B \in S^*$ is said to be a \emph{factor} of $\omega$ if there exist $m,n$ such that $B = \omega[m,n]$. In this case, $B$ is said to \emph{occur} in the \emph{interval} $[m,n] = \{m+i: i = 0,1,\dots,n - m\}$. The set of all factors of the infinite word $\omega$ is $\mathcal{F}(\omega)$, and the set of all factors of $\omega$ with length $n$ is $\mathcal{F}^n(\omega)$. 

An infinite word $\omega$ is said to be \emph{recurrent} if all of its factors occur infinitely often. If there exists, for all $n \in \mathbb{N}$, a positive integer $K(n)$ such that every $B' \in \mathcal{F}^n(\omega)$ is a factor of every $B \in \mathcal{F}^{K(n)}(\omega)$, then $\omega$ is \emph{uniformly recurrent}. 

\begin{defn}
For $k \in \mathbb{N}$, an \emph{additive $k$-power} is a word $\omega_1\omega_2\cdots \omega_k \in S^*$ such that $|\omega_1| = |\omega_2| = \cdots = |\omega_k|$ and $\sum \omega_1 = \sum \omega_2 = \cdots = \sum \omega_k$. An infinite word $\omega$ \emph{contains} an additive $k$-power if some factor of $\omega$ is an additive $k$-power.
\end{defn}

An infamous problem in combinatorics on words is that of whether or not there exists an infinite word over some finite subset of $\mathbb{Z}$ containing no additive square (i.e., additive $2$-power). This problem, which is to this day unsolved, appears in print as early as 1992 in a paper of Pirillo and Varricchio \cite{pirillo94} in the language of semigroups; independently, Halbeisen and Hungerb\"{u}hler \cite{halbeisen00} asked the question in the more combinatorial form described here. A number of articles have been inspired by this problem \cite{robertson11, cassaigne, freedman, gryt}. In particular, Cassaigne, Currie, Schaeffer, and Shallit \cite{cassaigne} show that there is an infinite word over $\{0,1,3,4\}$ with no additive cube (i.e., additive $3$-power).

Motivated by this problem of avoidability of additive squares, Ardal, Brown, Jungi\'{c}, and Sahasrabudhe \cite{ardal12} introduced the notion of additive complexity as follows.

\begin{defn}
Let $\omega$ be an infinite word over the alphabet $S$. The \emph{additive complexity} of $\omega$ is the map $\rho_{\omega}^{\Sigma}: \mathbb{N} \mapsto \mathbb{N}$  given by 
\[ \rho_{\omega}^{\Sigma}(n) = \left| \left\{ \sum B: B \in \mathcal{F}^n(\omega)\right\} \right|. \]
Thus, $\rho_{\omega}^{\Sigma}(n)$ is the number of distinct sums exhibited by factors of $\omega$ with length $n$.
We say that $\omega$ has \emph{bounded} additive complexity if there exists $M \in \mathbb{N}$ such that $\rho_{\omega}^{\Sigma}(n) \leq M$ for all $n \in \mathbb{N}$.
\end{defn}

The definition of the additive complexity of an infinite word is in the same spirit as the {\it abelian complexity} \cite{richomme09} of a word $\omega$, which is the function defined on $\mathbb{N}$ that, for $n\in \mathbb{N}$, counts the maximum number of factors of length $n$ such that no two are permutations of one another.  While it is clear that bounded abelian complexity implies bounded additive complexity, Ardal et al. \cite{ardal12} give a simple example to show that the converse is false. 

Ardal et al. \cite{ardal12} also proved the following theorem, connecting boundedness of additive complexity with the existence of additive $k$-powers.

\begin{thm} \label{thm: 1} 
If $\omega$ is an infinite word over $S$ with bounded additive complexity, then $\omega$ contains an additive $k$-power for all $k \in \mathbb{N}$.
\end{thm}

As an example of infinite words with bounded additive complexity we mention the family of \emph{balanced words}, i.e., words $\omega \in S^\mathbb{N}$ such that for all $n \in \mathbb{N}$, any two $B,B' \in \mathcal{F}^n(\omega)$ and any $s\in S$, one has $||B|_s-|B'|_s|\leq 1$. See the remark after Theorem~\ref{thm: 2}.

\subsection{Overview} \label{sub: 2}

In Section~\ref{sec: 2}, we obtain a new characterization of words with bounded additive complexity in terms of the slope of a word, and examine necessary and sufficient conditions for the slope of a word with bounded additive complexity to be a rational number. We develop results pertaining to the factorization of words with bounded additive complexity and rational slope, and prove a closure property for these words involving the contraction of certain intervals.

Section~\ref{sec: 3} is concerned with the relationship between morphisms (i.e., concatenation preserving mappings) and boundedness of additive complexity. More specifically, we explicitly describe the morphisms which map all words to words with bounded additive complexity. We also describe the words whose image under any morphism has bounded additive complexity. Using the theory developed in this section, we conjecture a way to explicitly construct all words with bounded additive complexity and rational slope. 

In Section~\ref{sec: 4} we construct, for every $k \in \mathbb{N}$, a recurrent word $\omega \in \{0,1,\dots,2k\}^\mathbb{N}$ having constant additive complexity (i.e., $\rho_{\omega}^{\Sigma}(n) = 2k + 1$ for all $n \in \mathbb{N}$).

Finally, Section~\ref{sec: 5} discusses a generalization to complexities defined by an arbitrary morphism $S^* \mapsto \mathbb{Z}^t$, $t \geq 1$.

\section{Slopes of Words and Bounded Additive Complexity} \label{sec: 2}

For a finite word $\omega$, the \emph{slope} of $\omega$ is defined as
\[ \slope{\omega} = \frac{1}{|\omega|} \sum \omega. \]
The \emph{slope} of an infinite word $\omega$ is
\[ \slope{\omega} = \lim_{n \to \infty} \frac{1}{n} \sum \omega[1,n] = \lim_{n \to \infty} \slope{\omega[1,n]}, \]
provided that this limit exists. For a word $\omega$, we may visualize each image $\omega(n)$ as the point $\left( n,\sum \omega[1,n]\right) \in \mathbb{Z}^2$ in the plane. Saying that $\slope{\omega}$ exists is then, intuitively, the same as saying that this set of points has a line of best fit passing through the origin.  
It is an easy exercise to construct a word $\omega$ for which $\slope{\omega}$ does not exist. 

The fact \cite[p.\ 50]{lothaire} that any two nonempty factors $B$ and $B'$ of a balanced binary word $\omega \in \{0,1\}^\mathbb{N}$ satisfy
$$ |\slope{B} - \slope{B'}| < \frac{1}{|B|} + \frac{1}{|B'|}$$  implies (since then the sequence $\{\slope{\omega[1,n]}\}_{n \in \mathbb{N}}$ is Cauchy) that
$\slope{\omega}$ exists for all balanced binary words. One thereby derives that for all nonempty factors $B$ of $\omega$,
\[ | \slope{B} - \slope{\omega} | \leq \frac{1}{|B|}. \]
In fact, a balanced binary word $\omega$ is eventually periodic (i.e., there are $A,B \in \{0,1\}^*$ such that $\omega = ABBB \cdots$) if and only if $\slope{\omega} \in \mathbb{Q}$ \cite[p.\ 52]{lothaire}. 

In the rest of this section we will prove generalizations of the aforementioned results, using words with bounded additive complexity in place of balanced binary  words. 

\subsection{Slope of a Word With Bounded Additive Complexity} \label{sub: 2.1}

We start by noticing that there are infinite words whose slope exists, yet do not have bounded additive complexity. For, consider the word $\omega \in \{0,1\}^\mathbb{N}$, defined by the rule for all $i \in \mathbb{N}$, that $\omega(i) = 1$ if and only $i$ is a power of $2$. Since $\omega$ contains arbitrarily long factors consisting of only $0$s, and yet $1$ occurs infinitely many times, $\omega$ does not have bounded additive complexity. But clearly, $\slope{\omega} = 0$. 

Our main tool in this section will be  the following theorem of Ardal et al. \cite{ardal12}.

\begin{thm} \label{thm: 2}
The infinite word $\omega \in S^\mathbb{N}$ has bounded additive complexity if and only if there exists $M \in \mathbb{N}$ such that for all $n \in \mathbb{N}$ and all $B,B' \in \mathcal{F}^n(\omega)$,  
\[ \left|\sum B - \sum B' \right| < M . \]
\end{thm}

Note that this condition holds for balanced words over $S$ if $M>\sum _{s\in S}|s|$. Further, we have the following lemma from Theorem~\ref{thm: 2}, which generalizes one of the aforementioned properties of balanced binary words.

\begin{lemma} \label{lemma: 1}
Let $\omega \in S^\mathbb{N}$ have bounded additive complexity and $M$ be as in Theorem~\ref{thm: 2}. Then, for any two nonempty factors $B$ and $B'$ of $\omega$,
\[ |\slope{B} - \slope{B'}| < \frac{M}{|B|} + \frac{M}{|B'|} . \]
\end{lemma}
 
\begin{proof} We use induction on $|B| + |B'|$. If $|B| = |B'|$, the claim of the lemma is a restatement of Theorem~\ref{thm: 2}, whence the base case $|B| + |B'| = 2$ holds. For the inductive step, assume without loss that $|B| > |B'|$, and put $B = CD$ with $|C| = |B'|$. Since $|D|+|B'|<|B|+|B'|$, the induction hypothesis gives 
\[ |\slope{D} - \slope{B'}| < \frac{M}{|D|} + \frac{M}{|B'|}. \]
Now,
\begin{align*} |\slope{B} - \slope{B'}| &= \left| \frac{|D|}{|B|}\slope{D} + \frac{|C|}{|B|} \slope{C} - \slope{B'} \right| \\
&\leq \frac{|D|}{|B|}| \slope{D} - \slope{B'}| + \frac{|C|}{|B|}| \slope{C} - \slope{B'}| \\
&< \frac{|D|}{|B|} \left( \frac{M}{|D|} + \frac{M}{|B'|} \right) + \frac{|C|M}{|B||B'|}
= \frac{M}{|B|} + \frac{M}{|B'|},  
\end{align*}
as desired.
\end{proof}

In the following theorem we characterize words with bounded additive complexity in terms of how quickly they approach their slope. Again, we note the analogy with balanced binary words.

\begin{thm} \label{thm: 3}
Let $\omega \in S^\mathbb{N}$. Then $\rho_{\omega}^{\Sigma}(n)$ is bounded if and only if $\slope{\omega}$ exists and there is $M \in \mathbb{N}$ such that for any nonempty factor $B$ of $\omega$, we have
\[ |\slope{B} - \slope{\omega}| \leq \frac{M}{|B|}. \]
\end{thm} 

\begin{proof} Assume that $\omega$ has bounded additive complexity and let $M$ be as in Theorem~\ref{thm: 2}. Lemma~\ref{lemma: 1} shows that the sequence $\{\slope{\omega[1,n]} \}_{n=1}^\infty$ is Cauchy, and hence $\slope{\omega}$ exists. 

Let $\varepsilon > 0$ be given. Select $N \in \mathbb{N}$ with $| \slope{\omega[1,n]} - \slope{\omega}| < \varepsilon$ and $M < n\varepsilon$, for all integers $n \geq N$. By Lemma~\ref{lemma: 1}, 
\begin{align*} |\slope{B} - \slope{\omega}| &\leq | \slope{B} - \slope{\omega[1,n]}| + |\slope{\omega[1,n]} - \slope{\omega}| \\
&< \frac{M}{|B|} + \frac{M}{n} + \varepsilon < \frac{M}{|B|} + 2\varepsilon. \end{align*}
Since $\varepsilon>0$ was arbitrary, we are done.

Conversely, suppose that $\slope{\omega}$ exists and there is some $M \in \mathbb{N}$ such that for every nonempty factor $B$ of $\omega$,
\[ |\slope{B} - \slope{\omega}| \leq \frac{M}{|B|}. \]
Then for every factor $B'$ of $\omega$ with $|B| = |B'|$,
\begin{align*} \left|\sum B - \sum B' \right| &\leq \left|\sum B - |B|\slope{\omega}\right| + \left| \sum B' - |B'|\slope{\omega} \right| \leq 2M, \end{align*}
and the result follows from Theorem~\ref{thm: 2}. 
\end{proof}

\subsection{Comparing Factors of Two Words with Bounded Additive Complexity} \label{sub: 2.2}

As a first application of Theorem~\ref{thm: 3}, we show that if two words with bounded additive complexity have differing slopes, then their factor sets must have a finite intersection.

\begin{lemma} \label{lemma: 2}
Suppose that $\omega_1 \in S^\mathbb{N}$ and $\omega_2 \in T^\mathbb{N}$ have bounded additive complexity. If $\slope{\omega_1} \not= \slope{\omega_2}$, there exists $M \in \mathbb{N}$ such that $\slope{B_1} \not= \slope{B_2}$ whenever $B_1 \in \mathcal{F}(\omega_1), B_2 \in \mathcal{F}(\omega_2)$ and $|B_1|,|B_2| \geq M$. 
\end{lemma}

\begin{proof}
Suppose $\slope{\omega_1} \not= \slope{\omega_2}$, and use Theorem~\ref{thm: 3} to get positive integers $M_1$ and $M_2$ large enough that 
\[ \left| \slope{B_1} - \slope{\omega_1} \right| \leq \frac{M_1}{|B_1|} \qquad \forall B_1 \in \mathcal{F}(\omega_1) \]
and 
\[ \left| \slope{B_2} - \slope{\omega_2} \right| \leq \frac{M_2}{|B_2|} \qquad \forall B_2 \in \mathcal{F}(\omega_2). \]
Select $M \in \mathbb{N}$ so that
\[ \max\left\{ \frac{M_1}{M}, \frac{M_2}{M} \right\} < \frac{1}{2}\left| \slope{\omega_1} - \slope{\omega_2} \right|, \]
and let $B_1 \in \mathcal{F}(\omega_1)$ and $B_2 \in \mathcal{F}(\omega_2)$ be given with $|B_1|,|B_2| \geq M$. Then $\slope{B_1}$ cannot equal $\slope{B_2}$, or else
\begin{align*} \left|\slope{\omega_1} - \slope{\omega_2} \right| &= \left| \slope{\omega_1} - \slope{B_1} + \slope{B_2} - \slope{\omega_2} \right| \\
&\leq \left| \slope{\omega_1} - \slope{B_1} \right| + \left| \slope{B_2} - \slope{\omega_2} \right| \\
&\leq \frac{M_1}{M} + \frac{M_2}{M} < \left| \slope{\omega_1} - \slope{\omega_2} \right|,
\end{align*}
a contradiction.
\end{proof}

Since equal words have equal slopes, the following is immediate. 

\begin{cor} 
Under the hypotheses of Lemma~\ref{lemma: 2}, the set $\mathcal{F}(\omega_1) \cap \mathcal{F}(\omega_2)$ is finite.
\end{cor}

The converse to the above corollary is false. To see this, we consider the periodic words $\omega_1,\omega_2 \in \{0,1\}^\mathbb{N}$ given by $\omega_1 = 010101 \cdots$ and $\omega_2 = 001100110011 \cdots$. It is readily checked that $\omega_1,\omega_2$ both have bounded additive complexity and slope equal to $\frac{1}{2}$; however, $0101$ is a factor of every sufficiently long factor of $\omega_1$, whereas $0101 \notin \mathcal{F}(\omega_2)$. Hence $\mathcal{F}(\omega_1) \cap \mathcal{F}(\omega_2)$ is finite, even though $\slope{\omega_1} = \slope{\omega_2}$. 

We have as of yet been unable to prove or disprove the converse of Lemma~\ref{lemma: 2}. Therefore we leave this, and a related question, as open problems. 

\begin{q} Prove or disprove the converse of Lemma~\ref{lemma: 2}.
\end{q}

\begin{q} Let $\omega_1 \in S^\mathbb{N}$ and $\omega_2 \in T^\mathbb{N}$ have bounded additive complexity, and suppose further that $\slope{\omega_1} = \slope{\omega_2}$. Under what conditions are there infinitely many pairs $(B_1,B_2) \in \mathcal{F}(\omega_1) \times \mathcal{F}(\omega_2)$ such that $|B_1| = |B_2|$ and $\slope{B_1} = \slope{B_2}$? 
\end{q}

\subsection{Rationality of Slope} \label{sub: 2.3}

In this section we are concerned with necessary and sufficient conditions for the slope of a word $\omega$ with bounded additive complexity to be rational.
We observe that Sturmian words, as balanced binary words, are words with bounded additive complexity. Hence, bounded additive complexity does not imply rationality of slope \cite[ch.\ 2]{lothaire}.

Recalling that by Theorem~\ref{thm: 1}, $\omega$ contains an additive $k$-power for every $k \in \mathbb{N}$, we first have the following.

\begin{thm} \label{thm: 4}
Let $\omega \in S^\mathbb{N}$ have bounded additive complexity and assume that $ \slope{\omega} = p/q$, where  $p$ and $q\geq1$ are relatively prime integers. There exists $K \in \mathbb{N}$ such that any additive $k$-power in $\omega$ with $k \geq K$ has slope $p/q$.
\end{thm}

\begin{proof}
By Theorem~\ref{thm: 2}, there exists $M \in\mathbb{N}$ such that if $B_1$ and $B_2$ are factors of $\omega$ with equal lengths, then $|\sum B_1 - \sum B_2| \leq M$. Let $j = M + 1$ and let $J = J_1J_2\cdots J_{jq}$ be any additive $jq$-power in $\omega$. By Theorem~\ref{thm: 3},
\begin{align*} \frac{M}{jq|J_1|} \geq \left| \slope{J} - \slope{\omega} \right| = \left| \frac{jq \sum J_1}{jq |J_1|} - \frac{p}{q} \right| = \frac{1}{q|J_1|}\left| q\sum J_1 - p|J_1| \right|, \end{align*}
whence
\[ 1 > \frac{M}{j} \geq \left|q \sum J_1 - p |J_1| \right|. \]
Since $q\sum J_1$ and $p |J_1|$ are integers,  $q \sum J_1 = p|J_1|$. Thus 
\[ \slope{J} = \frac{1}{|J_1|}\sum J_1 = \frac{p}{q} = \slope{\omega}, \]
and consequently every additive $k$-power in $\omega$ with $k \geq jq$ has slope equal to $p/q$.
\end{proof}

Next we give a necessary and sufficient condition for rationality of the slope of a word with bounded additive complexity.

\begin{lemma} \label{lemma: 3}
Let  $\omega \in S^\mathbb{N}$ be a word with bounded additive complexity. Then $\slope{\omega}$ is rational if and only if there exists some infinite collection $\mathcal{C}$ of factors of $\omega$ such that for any $B,B' \in \mathcal{C}$, $\slope{B} = \slope{B'}$.
\end{lemma}

\begin{proof} 
If $\slope{\omega}$ is rational, the existence of such a collection $\mathcal{C}$ follows from Theorem~\ref{thm: 4}.

Conversely, assume such a collection $\mathcal{C}$ exists. We let $\{B_i: i \in \mathbb{N}\} \subset \mathcal{C}$ be such that $|B_i| \geq i$, $i \geq 1$. Theorem~\ref{thm: 3} gives $M \in \mathbb{N}$ such that for any $i$,
\begin{align*} |\slope{B_1} - \slope{\omega}| &\leq |\slope{B_1} - \slope{B_i}| + |\slope{B_i} - \slope{\omega}| \\
&= |\slope{B_i} - \slope{\omega}| \leq \frac{M}{|B_i|} \leq \frac{M}{i}. 
\end{align*}
Hence,  $\slope{B_1} = \slope{\omega}$. In particular, $\slope{\omega}$ is rational. 
\end{proof}

Lemma~\ref{lemma: 3} is perhaps better re-phrased as a statement about words $\omega$ with bounded additive complexity and irrational slope. In particular, it implies that $\omega$ has only finitely many factors of slope equal to any number. One might then surmise that the number of such factors has a uniform bound independent of the given rational; however, our next result implies that this is not the case. 

\begin{cor} Let $\omega \in S^\mathbb{N}$ have bounded additive complexity. Then

(i) $\slope{\omega}$ is irrational if and only if for every $\alpha \in \mathbb{Q}$, the set
\[ \mathcal{F}_{\alpha}(\omega) = \left\{B \in \mathcal{F}(\omega): \slope{B} = \alpha \right\} \]
is finite; 

(ii) for every $k \in \mathbb{N}$, there exists $\alpha \in \mathbb{Q}$ such that $|\mathcal{F}_\alpha(\omega)| > k$.
\end{cor}

\begin{proof} The statement (i) follows directly from Lemma~\ref{lemma: 3}, and by a theorem of Brown \cite{brown12}, (ii) holds for any infinite word over a finite set of integers. 
\end{proof}

Let $\omega \in S^\mathbb{N}$ have bounded additive complexity with $\slope{\omega} =  p/q$, for some relatively prime $p,q \in \mathbb{Z}$. We will associate with $\omega$ a finite coloring
of the natural numbers which we will use to provide us with better insight into the structure of $\omega$. First, we use Theorem~\ref{thm: 3} to get $M \in \mathbb{N}$ such that
$$\left| \sum \omega[1,n] - \frac{np}{q} \right| \leq M$$
for all $ n \in \mathbb{N}$.
We define the coloring $\chi_\omega: \mathbb{N} \mapsto [-M,M]$  by
\[ \chi_\omega(m) = \sum \omega[1,mq] - mq \cdot \slope{\omega}=\sum \omega[1,mq] - mp.\] 

The coloring $\chi_{\omega}$ has the following property.

\begin{lemma} \label{lemma: 4} Let $\omega \in S^\mathbb{N}$ have bounded additive complexity with $\slope{\omega}= p/q$ for some relatively prime $p,q \in \mathbb{Z}$. Then for  $m>n$, $\chi_\omega(m) = \chi_\omega(n)$ if and only if $\slope{\omega[nq + 1, mq]} = p/q$.
\end{lemma}

\begin{proof} Since $\chi_\omega(m) = \chi_\omega(n)$ if and only if
\[ \sum \omega[1,mq] - mp = \sum \omega[1,nq] - np, \]
it follows that for $m>n$, $\chi_\omega(m) = \chi_\omega(n)$ if and only if
\[ \slope{\omega[nq + 1, mq]} = \frac{1}{mq - nq} \sum \omega[nq + 1, mq] = \frac{p}{q}. \qedhere \]
\end{proof}

The next result may be viewed as a stronger version of Theorem~\ref{thm: 1}, under the additional hypothesis that the slope of the word in question is rational.

\begin{thm} \label{thm: 5} 
Let $\omega \in S^\mathbb{N}$ have bounded additive complexity with $\slope{\omega} \in \mathbb{Q}$. For every $k,\ell \in \mathbb{N}$, there exists $M(k,\ell) \in \mathbb{N}$ such that every $B \in \mathcal{F}(\omega)$ of length $M(k,\ell)$ contains an additive $\ell$-power $B_1B_2\ldots B_\ell$ such that $k$ divides $|B_1|$ and $\slope{B_i}=p/q$, $i=1,2,\ldots,\ell$.
\end{thm}

\begin{proof}
Let $k$ and $\ell$ be given positive integers. By van der Waerden's Theorem \cite{vdW} applied to the set $k\mathbb{N}$ there exists, for any positive integer $n$, a positive integer $N=N(n,k,\ell)$ such that any $n$-coloring of $N$ consecutive positive integers contains a monochromatic $(\ell+1)$-term arithmetic progression with gap a multiple of $k$.

Let $\slope{\omega} = p/q$, for some relatively prime $p,q \in \mathbb{Z}$, and let $M$ and the coloring $\chi_\omega: \mathbb{N} \mapsto [-M,M]$ be as above. Let $M(k,\ell)=q(N(2M+1,k,\ell)+1)$, let $m \geq 1$ be arbitrary, and let $B=\omega[m,m+M(k,\ell)-1]$. Let $t\in \mathbb{N}$ be such that $m\in ((t-1)q,tq]$. Then, from $tq<m+q$, it follows that
\begin{align*}
q(t+N(2M+1,k,\ell)) &=tq+qN(2M+1,k,\ell)\leq m+q+qN(2M+1,k,\ell) - 1 \\
&= m + M(k,\ell) - 1,
\end{align*}
and we conclude that $\{jq:t\leq j\leq t+N(2M+1,k,\ell)\}\subset [m,m+M(k,\ell)-1]$. Let $\{a+idk:i=0,1,\ldots \ell\}$ be a $\chi_\omega$-monochromatic $(\ell+1)$-term arithmetic progression with gap $dk$, for some natural number $d$, that is contained in the interval $[t, t+N(2M+1,k,\ell)-1]$. For factors $B_i=\omega[q(a+(i-1)dk)+1,q(a+idk)]$, $i=1,\ldots,\ell$,  we have that $|B_i|=dkq$ and, by Lemma~\ref{lemma: 4}, $\slope{B_i}=p/q$.
\end{proof}

\subsection{Factorizations of Words with Bounded Additive Complexity and Rational Slope} \label{sub: 2.4}

Recall that a balanced binary word is eventually periodic if and only if its slope is rational. The following result shows that a generalization of this property holds for words with bounded additive complexity. 

\begin{thm} \label{thm: 6}
Let  $\omega \in S^\mathbb{N}$ have bounded additive complexity. Then $\slope{\omega} \in \mathbb{Q}$ if and only if $\omega = AB_1B_2B_3 \cdots$, where $\{B_i\}_{i \in \mathbb{N}}$ is a sequence of nonempty factors of $\omega$ such that $\slope{B_1} = \slope{B_i}$ for all $i \in \mathbb{N}$. Moreover, $\slope{B_1} = \slope{\omega}$.
\end{thm}

\begin{proof} Assuming that $\omega = AB_1B_2B_3 \cdots$ with $\slope{B_1} = \slope{B_i}$, for all $i \in \mathbb{N}$, we deduce that $\slope{\omega}$ is rational by Lemma~\ref{lemma: 3}.

On the other hand, assume $p/q = \slope{\omega} \in \mathbb{Q}$ for some relatively prime $p,q \in \mathbb{Z}$. By the pigeonhole principle, there exists an increasing sequence $\{x_i\}_{i=1}^\infty$ of positive integers which is monochromatic under $\chi_\omega$. By Lemma~\ref{lemma: 4}, the proof is complete with $A = \omega[1,x_1q]$, and for all $i \in \mathbb{N}$, $B_i = \omega[x_iq + 1, x_{i+1}q]$.
\end{proof}

For $\omega \in S^\mathbb{N}$ with bounded additive complexity and rational slope, there may not exist a factorization as in Theorem~\ref{thm: 6} where $\{|B_i|\}_{i \in \mathbb{N}}$ is bounded, even if $\omega$ is recurrent. We construct an example to illustrate this point.

To begin the construction, let $X_n' \in [1,n]^*$ denote the finite word obtained from concatenating all elements of $[1,n]^n$ in lexicographical order, from least to greatest, and consider the word $\omega' = X_1'X_2'X_3' \cdots \in \mathbb{N}^\mathbb{N}$. 

We define the word $\omega = X_1X_2X_3 \cdots$, where 
\[ X_n = \begin{cases} 01^{\omega'(n)}2, &\textup{if $\omega'(n)$ is odd;} \\ 
21^{\omega'(n)} 0, &\textup{otherwise.} \end{cases} \]
The word $\omega$ is recurrent since $\omega'$ is recurrent.
To show that $\omega$ has bounded additive complexity, we let $B \in \mathcal{F}(\omega)$ be given, and observe that $B$ can be written in the form $B = Y_iX_{i+1}X_{i+2} \cdots X_{i + k} Y_{i + k + 1}$, where $Y_i$ is a suffix of $X_i$ and $Y_{i + k + 1}$ is a prefix of $X_{i + k + 1}$. Hence
\[ \sum B = \sum Y_i + \sum Y_{i + k + 1} + |X_{i+1}X_{i+2} \cdots X_{i + k}|, \]
where $|Y_i| - 1 \leq \sum Y_i \leq |Y_i| + 1$ and $|Y_{i+k}| - 1 \leq \sum Y_{i+k} \leq |Y_{i+k}| + 1$, and this implies that $|B| - 2 \leq \sum B \leq |B| + 2$.
Therefore, if $B' \in \mathcal{F}^{|B|}(\omega)$ is given, then $\left| \sum B - \sum B' \right| \leq 4$.
By Theorem~\ref{thm: 2}, $\omega$ has bounded additive complexity. Moreover, $\slope{\omega} = 1 \in \mathbb{Q}$ by Theorem~\ref{thm: 6}, since $\slope{X_i} = 1$ for all $i \in \mathbb{N}$.

All that remains is to show that $\omega$ cannot be written as $\omega = AB_1B_2B_3 \cdots$, where $\slope{B_i} = 1$ for all $i \in \mathbb{N}$ and there exists $M$ such that $0 < |B_i| < M $ for all $i \in \mathbb{N}$. For this, consider, for a given $n \in \mathbb{N}$, the sequence $\{k_{n,i}\}_{i \in \mathbb{N}}$ defined by the following rule: $k_{n,1}$ is the least positive integer such that $\omega[n,k_{n,1}]$ has slope $1$, and for $i > 1$, $k_{n,i}$ is the least positive integer such that $\omega[k_{n,i-1} + 1,k_{n,i}]$ has slope $1$. Then it is enough to show, for all $n \geq 1$, that the sequence $\{k_{n,i}\}_{i \in \mathbb{N}}$ does not have bounded gaps. For this, note that $001^m221^{m+1}00$ occurs infinitely many times, for all odd $m \in \mathbb{N}$. Let $001^m221^{m+1}00 = \omega[\ell + 1, \ell + 2m + 7]$ with $\ell > k_{n,1}$, and assume for a contradiction that $\{k_{n,i}\}_{i \in \mathbb{N}}$ has no gap of size at least $m$. Thus, for some $i \in \mathbb{N}$, one of the following holds:

(a) $k_{n,i} = \ell$; 

(b) $k_{n,i} = \ell + 1$;

(c) $k_{n,i} = \ell + 2$;

(d) $k_{n,i} = \ell + h$, for some $h \in [3,m-1]$.

In case (a), $k_{n,i+1} = \ell + m + 4$, giving a gap of size $m + 4$. In case (b), $k_{n,i+1} = \ell + m + 3$, giving a gap of size $m + 2$. In case (c), $k_{n,i + j} = k_{n,i + j - 1} + 1$ for all $j = 1,2,\dots,m$, whence $k_{n,i + m} = \ell + m + 2$. Moreover, $k_{n,i + m + 1} = \ell + 2m + 7$, giving a gap size of $m + 5$. In case (d), $k_{n,i + j} = k_{n,i + j - 1} + 1$ for $j=1,2,\dots,m - h + 2$, and so $k_{n,i + m - h + 2} = \ell + m + 2$. As in case (c), $k_{n,i + m - h + 3} = \ell + 2m + 7$, and we obtain a gap size of $m + 5$. Since $m$ was an arbitrary odd positive integer, and all cases lead to a contradiction, the proof is complete.

The above example should be contrasted with the following result.

\begin{thm} \label{thm: 7}
Let $\omega \in S^\mathbb{N}$ be a uniformly recurrent word with bounded additive complexity, and suppose that $\slope{\omega} \in \mathbb{Q}$. Then $\omega$ admits a factorization of the form $\omega = AB_1B_2B_3\cdots$, where $\{B_i\}_{i \in \mathbb{N}}$ is a sequence of nonempty factors of $\omega$ such that: 

(i) $\slope{B_i} = \slope{\omega}$ for all $i \in \mathbb{N}$;

(ii) $\{|B_i|\}_{i \in \mathbb{N}}$ is bounded.
\end{thm}

\begin{proof}
For a given factor $B$ of $\omega$, let $d(B) = \sum B - |B|\slope{\omega}$.
Since $\slope{\omega}$ is rational, Theorem~\ref{thm: 3} lets us pick a particular $B \in \mathcal{F}(\omega)$ such that $|d(B)|$ is as large as possible, and we assume without loss that $\sum B > \slope{\omega} |B|$. Since $\omega$ is uniformly recurrent, there is a least positive integer $K = K(|B|)$ such that every factor of length at least $K$ has $B$ as a factor. 
Suppose $\omega[i+1,j] = B$, select the minimum $i' \in [j+2, j + K + 1]$ such that $B = \omega[i' +1,i' + j - i]$, and let $A = \omega[j+1,i']$. 

We claim that $BA = \omega[i + 1, i']$ has slope equal to that of $\omega$. If this claim holds, then the above argument may be repeated with $B = \omega[i' + 1, i' + j - 1]$ in place of $B = \omega[i + 1, j]$, and since $|BA| < |B| + K + 2$ the result follows. Thus, it suffices to prove the claim, and we proceed to this end. 
Observe that maximality of $d(B)$ implies $d(B) \geq  d(BAB)  = 2d(B) + d(A)$.
In other words $d(A) \leq - d(B)$,
and since $|d(B)|$ is maximum, $d(A) = - d(B)$.
Hence $d(BA) = d(B) + d(A) = 0$,
so that $\slope{BA} = \slope{\omega}$.
\end{proof}

\subsection{Contraction} \label{sub: 2.5}

The results developed in the last two subsections suggest that, in a word $\omega \in S^\mathbb{N}$ with bounded additive complexity and rational slope $p/q$, there is an abundance of intervals $[m,n]$ such that $\slope{\omega[m,n]} = p/q$. It is tempting, then, to wonder how the removal of some of these intervals affects the slope of $\omega$. The next result shows that after contracting (i.e., removing) any set $\mathcal{I}$ of intervals such that $\slope{\omega[m,n]} = p/q$ for every $[m,n] \in \mathcal{I}$, we are left with a word with bounded additive complexity and slope $p/q$.

If $\mathcal{I}$ is a set of intervals, its elements are \emph{separated} if whenever $[i,j], [k,\ell] \in \mathcal{I}$ with $i \leq k$, we have $j + 1 < k$.

\begin{defn} 
Given a word $\omega \in S^\mathbb{N}$ and $i,j \in \mathbb{N}$, we can form a word $\omega'$ given by $\omega'(k) = \omega(k)$ for $k < i$ and $\omega'(k) = \omega(k + j -i + 1)$ for $k \geq i$. We say that $\omega'$ is obtained from $\omega$ by \emph{contraction of $[i,j]$}, and we write $\omega' = \omega \setminus [i,j]$. We extend this definition in the natural way to contraction of any set $\mathcal{I}$ of separated intervals. Similarly, if $\omega'$ is the word obtained from $\omega$ via contraction of $\mathcal{I}$, we will write $\omega' = \omega \setminus \mathcal{I}$.
\end{defn}

\begin{thm} \label{thm: 8}
Let $\omega \in S^\mathbb{N}$ be a word with bounded additive complexity, and suppose $p/q = \slope{\omega} \in \mathbb{Q}$. If $\mathcal{I}$ is a set of separated intervals such that $\slope{\omega[m,n]} = p/q$ for all $[m,n] \in \mathcal{I}$, then $\omega \setminus \mathcal{I}$ has bounded additive complexity and $\slope{\omega \setminus \mathcal{I}} = p/q$.
\end{thm}

\begin{proof}
Since every suffix of $\omega$ also has bounded additive complexity and slope $p/q$, we may assume that $\mathcal{I}$ is infinite. Hence, $\mathcal{I} = \{[i_k,j_k]: k \in \mathbb{N}\}$, for some $j_k \geq i_k > j_{k-1} + 1 \geq i_{k-1}$, $k \geq 1$. We put $B_k = \omega[i_k,j_k]$.

Let $A_1 = \omega[1,i_1 - 1]$ if $i_1 > 1$, and otherwise let $A_1$ be empty. For $k > 1$, let $A_k$ be the nonempty word such that $\omega[i_{k-1}, j_k] = B_{k-1}A_kB_k$. Thus $\omega = A_1B_1A_2B_2A_3B_3 \cdots$, and we must show that $\omega \setminus \mathcal{I} = A_1A_2A_3 \cdots$ has bounded additive complexity and slope $p/q$. Let $A \in \mathcal{F}(\omega\setminus \mathcal{I})$ and write $A = A_{\ell}' A_{\ell + 1} \cdots A_{m-1} A_{m}'$, for some $\ell < m \in \mathbb{N}$, where $A_{\ell}'$ (\emph{resp.} $A_{m}'$) is a possibly empty prefix (\emph{resp.} suffix) of $A_{\ell}$ (\emph{resp.} $A_m$). Let $B = B_{\ell}B_{\ell + 1} \cdots B_{m-1}$. By Theorem~\ref{thm: 3} there exists $M \in \mathbb{N}$, independent of $A$, such that 
\begin{align*} M &\geq \left| \sum A_{\ell}'B_{\ell}A_{\ell + 1} B_{\ell + 1} \cdots B_{m-1} A_{m}' - \frac{p}{q}|A_{\ell}'B_{\ell}A_{\ell + 1} B_{\ell + 1} \cdots B_{m-1} A_{m}'| \right| \\
&= \left| \sum B + \sum A - \frac{p}{q} |A| - \frac{p}{q} |B| \right| = \left| \sum A - \frac{p}{q} |A| \right|, 
\end{align*}
whence 
\[ \left| \slope{A} - \frac{p}{q} \right| \leq \frac{M}{|A|}.  \]
A second application of Theorem~\ref{thm: 3} yields the desired result.
\end{proof}

\section{Bounded Additive Complexity and Morphisms} \label{sec: 3}

A map $\phi: S^\infty \mapsto T^\infty$ is called a \emph{morphism} if $\phi(BB') = \phi(B)\phi(B')$ for all $B \in S^*$ and $B' \in S^\infty$. Clearly, $\phi$ is uniquely determined by its assignments on the elements of $S$, and we say that $\phi$ is \emph{generated} by these assignments. We call $\phi$ a \emph{non-erasing} morphism if no nonempty word maps to the empty word.

Non-erasing morphisms receive special attention in combinatorics on words because, being compatible with the concatenation of words, they relate the form of the input word with the form of the output word. Below we investigate the effect of non-erasing morphisms on boundedness of additive complexity. 

\subsection{Anchors} \label{sub: 3.1}

We begin by introducing the following terminology. 

\begin{defn}
If $\phi: S^\infty \mapsto T^\infty$ is a non-erasing morphism, we say that $\phi$ \emph{anchors} $\omega \in S^\mathbb{N}$ if $\phi(\omega) \in T^\mathbb{N}$ has bounded additive complexity. If $\phi$ anchors every element of $S^\mathbb{N}$, then $\phi$ is said to be an \emph{anchor}. 
\end{defn}

In this section we give necessary and sufficient conditions for $\phi$ to be an anchor and an explicit arithmetic description of all anchors from $S^\infty$ to $T^\infty$. We note that, as a result of the following theorem, whether or not a morphism is an anchor may be checked in linear time depending on the size of $S$.

\begin{thm} \label{thm: 9}
Let $\phi: S^\infty \mapsto T^\infty$ be a non-erasing morphism. Then $\phi$ is an anchor if and only if $\slope{\phi(s)} = \slope{\phi(s')}$ for all $s,s' \in S$.
\end{thm}

The proof of Theorem~\ref{thm: 9} relies on the following two lemmas. In both lemmas we take that $S = \{s_1,s_2,\dots,s_k\}$  and that $\phi: S^\infty \mapsto T^\infty$ is a non-erasing morphism.

\begin{lemma} \label{lemma: 5}
The following statements are equivalent:

(i) $\slope{\phi(s)} = \slope{\phi(s')}$ for all $s,s'\in S$;

(ii) For every $B_1,B_2 \in S^*$, $|\phi(B_1)| = |\phi(B_2)|$ implies $\sum \phi(B_1) = \sum \phi(B_2)$.

\end{lemma}

\begin{proof} To see that $(ii)$ implies $(i)$, we prove the contrapositive. Suppose that $(i)$ does not hold. Let $s_i \in S$ be such that $\slope{\phi(s_1)} \not= \slope{\phi(s_i)}$ and put
\[ p = \frac{\lcm(|\phi(s_1)|,|\phi(s_i)|)}{|\phi(s_1)|}, \qquad q = \frac{\lcm(|\phi(s_1)|,|\phi(s_i)|)}{|\phi(s_i)|}. \]
Then 
\[ |\phi(s_1^p)| = p|\phi(s_1)| =
q|\phi(s_i)| = |\phi(s_i^q)|, \] 
but
\[ \sum \phi(s_1^p) = p\sum \phi(s_1) \ne q \sum \phi(s_i) = \sum \phi(s_i^q). \]
So $(ii)$ does not hold.

On the other hand, suppose $(i)$ holds, and let $B \in T^*$ have $B = \phi(B')$ for some $B' \in S^*$. Since 
\[ |B| = \sum_{1 \leq i \leq k} |\phi(s_i)||B'|_{s_i} = \sum_{1 \leq i \leq k} \frac{|\phi(s_1)|\sum \phi(s_i)}{\sum \phi(s_1)} |B'|_{s_i}, \]
we have
\[ \sum B = \sum_{1 \leq i \leq k} \left( |B'|_{s_i} \sum \phi(s_i) \right) = \frac{|B| \sum \phi(s_1)}{|\phi(s_1)|} = |B|\slope{\phi(s_1)}. \]
We conclude that $(ii)$ holds.
\end{proof}

\begin{lemma} \label{lemma: 6}
If condition $(ii)$ in Lemma~\ref{lemma: 4} holds, then 
\[ \max\left\{ \left| \sum \phi(B_1) - \sum \phi(B_2) \right| : B_1,B_2 \in S^*,||\phi(B_1)| - |\phi(B_2)|| \leq 2N - 2\right\} \]
exists for any fixed $N \in \mathbb{N}$.
\end{lemma}

\begin{proof} Let $B_1,B_2 \in S^*$ satisfy $||\phi(B_1)| - |\phi(B_2)|| \leq 2N - 2$.
If $\slope{\phi(B_1)} \ne \slope{\phi(B_2)}$,
putting
\[ p = \frac{\lcm(|\phi(B_1)|,|\phi(B_2)|)}{|\phi(B_1)|}, \qquad q = \frac{\lcm(|\phi(B_1)|,|\phi(B_2)|)}{|\phi(B_2)|} \]
yields $|\phi(B_1^p)| = |\phi(B_2^q)|$ yet $\sum \phi(B_1^p) \ne \sum \phi(B_2^q)$, contradicting our hypothesis. Thus,
\[ \sum \phi(B_1) - \frac{|\phi(B_1)|}{|\phi(B_2)|}\sum \phi(B_2) = 0. \]
We assume that $|\phi(B_2)| \geq |\phi(B_1)|$. Now,
\begin{align*} \left|\sum \phi(B_1) - \sum \phi(B_2)\right| &= \left|\frac{|\phi(B_1)|}{|\phi(B_2)|} - 1 \right|\left| \sum \phi(B_2)\right| \leq \left|\frac{|\phi(B_1)|}{|\phi(B_2)|} - 1 \right| |\phi(B_2)|\max S_2\\
&\leq (2N - 2)\max S_2. \qedhere \end{align*} 
\end{proof}

We are now ready to prove Theorem~\ref{thm: 9}.

\begin{proof}[Proof of Theorem~\ref{thm: 9}]
By Lemma~\ref{lemma: 5}, it suffices to prove that $\phi$ is an anchor if and only if $|\phi(B_1)| = |\phi(B_2)|$ implies $\sum \phi(B_1) = \sum \phi(B_2)$, for every $B_1,B_2 \in S^*$. 

We first assume this condition does not hold, and show that $\phi$ is not an anchor. Therefore,  let $B_1,B_2 \in S^*$ be such that $|\phi(B_1)| = |\phi(B_2)|$, but $\sum \phi(B_1) - \sum \phi(B_2) = k >0$. Then for $m \in \mathbb{N}$, $B_1^m,B_2^m$ satisfy 
\[ |\phi(B_1^m)| = m|\phi(B_1)| = m|\phi(B_2)| = |\phi(B_2^m)| \]
and 
\[ \sum\phi(B_1^m) - \sum\phi(B_2^m) = m\left(\sum \phi(B_1) - \sum\phi(B_2)\right) = mk \to \infty \] 
as $m \to \infty$. By Theorem~\ref{thm: 2}, the image of $\omega = B_1B_2B_1^2B_2^2B_1^3B_2^3 \cdots \in S^\mathbb{N}$
under $\phi$ does not have bounded additive complexity, and $\phi$ is not an anchor.

Conversely, we assume that the condition holds and show that $\phi$ is an anchor.  To this end, let $\omega \in S^\mathbb{N}$ and let $B_1,B_2 \in T^*$ be equally long factors of $\phi(\omega)$. Write $B_1 = U_1\phi(U_2)U_3$ and $B_2 = V_1\phi(V_2)V_3$, where $U_2,V_2 \in S^*$, $U_1,U_3,V_1,V_3 \in T^*$, and 
\[ |U_1|,|U_3|,|V_1|,|V_3| < N = \max_{s \in S_1} |\phi(s)|, \]
so that $||\phi(U_2)| - |\phi(V_2)|| \leq 2N - 2$.
By Lemma~\ref{lemma: 6}
\[ M = \max\left\{ \left| \sum \phi(U) - \sum \phi(V) \right| : U,V \in S^*,||\phi(U)| - |\phi(V)|| \leq 2N - 2\right\} \]
exists, and since $\{U \in T^*: |U| < N\}$ is finite, we may set
\[ M' = \max\left\{\left|\sum U - \sum V \right|: U,V \in T^*, |U|,|V|<N \right\} \]
to obtain 
\begin{align*} \left| \sum B_1 - \sum B_2 \right| &= \left| \left(\sum U_1 - \sum V_1 \right) + \left(\sum \phi(U_2) - \sum \phi(V_2)\right) + \left(\sum U_3 - \sum V_3\right) \right| \\
&\leq \left|\sum U_1 - \sum V_1\right| + \left|\sum \phi(U_2) - \sum \phi(V_2) \right| + \left|\sum U_3 - \sum V_3 \right| \\
&\leq 2M' + M.
\end{align*}
By Theorem~\ref{thm: 2}, the morphism $\phi$ is an anchor.
\end{proof}

If $\phi: S^\infty \mapsto T^\infty$ is an anchor, the map $a \mapsto \slope{\phi(a)}$ takes only one value $||\phi||$, and we call this value the \emph{weight} of $\phi$.

\begin{cor}
Let  $\phi: S^\infty \mapsto T^\infty$ be an anchor and $\omega \in S^\mathbb{N}$. Then $\slope{\phi(\omega)} = ||\phi||$.
\end{cor}

\begin{proof} By the definition of an anchor and by Theorem~\ref{thm: 3}, $\slope{\phi(\omega)}$ exists. Now, since $\slope{\phi(\omega[1,n])} = ||\phi||$ for all $n \in \mathbb{N}$, and since $\{\slope{\phi(\omega[1,n])}\}_{n\in \mathbb{N}}$ is a subsequence of $\{\slope{\phi(\omega)[1,n]}\}_{n\in \mathbb{N}}$, the result follows.
\end{proof}

Our next result is an extension of Theorem~\ref{thm: 9}, and it provides an explicit arithmetic description of the anchors of a given weight in terms of a matrix equation. 

Suppose $S = \{s_1,\dots,s_k\}$ and $T = \{t_1,\dots,t_\ell\}$, and let $\phi: S^\infty \mapsto T^\infty$ be a non-erasing morphism. We associate a matrix $\mathcal{M}(\phi) \in \mathbb{Z}^{k \times \ell}$ with $\phi$ by defining $\mathcal{M}(\phi)_{i,j} = |\phi(s_i)|_{t_j}$ for all $i=1,2,\dots,k$ and $j=1,2,\dots,\ell$. Also, with every $\alpha \in \mathbb{Q}$ we associate a column vector $\mathbf{t_\alpha} \in \mathbb{Q}^\ell$ by defining $(\mathbf{t_\alpha})_j = t_j - \alpha$ for all $j=1,2,\dots,\ell$.

\begin{cor}
The morphism $\phi: S^\infty \mapsto T^\infty$ is an anchor if and only if there exists $\alpha \in \mathbb{Q}$ such that $\mathcal{M}(\phi)\mathbf{t_\alpha} = \mathbf{0}$.
\end{cor}

\begin{proof}
Theorem~\ref{thm: 9} implies that $\phi$ is an anchor of weight $\alpha \in \mathbb{Q}$ if and only if
\[ \sum_{1 \leq j \leq \ell} \mathcal{M}(\phi)_{i,j} t_j = \sum \phi(s_i)  = \alpha |\phi(s_i)| = \alpha \sum_{1 \leq j \leq \ell} \mathcal{M}(\phi)_{i,j}, \]
for all $i=1,2,\dots,k$.
\end{proof} 

The following general problem has yet to be explored further.

\begin{q}
Let an infinite word $\omega \in S^\mathbb{N}$ and a morphism $\phi: S^\infty \mapsto T^\infty$ be given. Under what circumstances is it decidable whether or not $\phi$ anchors $\omega$?
\end{q}

\subsection{Preserving Bounded Additive Complexity Under Morphisms} \label{sub: 3.2}

In this section we characterize the words such that their image under any non-erasing morphism has bounded additive complexity. 

Earlier we mentioned the abelian complexity of an infinite word $\omega$, and at this time we present a formal definition thereof.

\begin{defn} Let $\omega \in S^\mathbb{N}$. The \emph{abelian complexity} of $\omega$ is the map $\rho_{\omega}^{\psi}: \mathbb{N} \mapsto \mathbb{N}$ that associates, with each $n \in \mathbb{N}$, the maximum number of distinct elements of $\mathcal{F}^n(\omega)$, no two of which are permutations of each other (the use of the symbol $\psi$ in this notation will be explained in Section~\ref{sec: 5}). 
The abelian complexity of $\omega$ is said to be \emph{bounded} if there exists $M \in \mathbb{N}$ such that $\rho_{\omega}^{\psi}(n) \leq M$ for all $n \in\mathbb{N}$.
\end{defn}

Richomme, Saari, and Zamboni \cite{richomme09} characterized words with bounded abelian complexity in the following way.

\begin{lemma} \label{lemma: 7}
Let  $\omega \in S^\mathbb{N}$. Then $\omega$ has bounded abelian complexity if and only if there exists $C \in \mathbb{N}$ such that for all $n \in \mathbb{N}$ and all $s \in S$, $||B|_s - |B'|_s| < C$ whenever $B,B' \in \mathcal{F}^n(\omega)$.
\end{lemma}

This fact will be our main tool in proving the following result.

\begin{thm} \label{thm: 10}
Let  $\omega \in S^\mathbb{N}$. The following are equivalent.

(i) The word $\omega$ has bounded abelian complexity;

(ii) For all finite $T \subset \mathbb{Z}$, every non-erasing morphism $\phi: S^\infty \mapsto T^\infty$ anchors $\omega$.
\end{thm}

\begin{proof}
First, assume $(i)$ holds. Then, for a finite $T \subset \mathbb{Z}$ and a non-erasing morphism $\phi: S^\infty \mapsto T^\infty$, a theorem of Cassaigne, Richomme, Saari, and Zamboni \cite{cassaigne12} asserts that $\phi(\omega)$ has bounded abelian complexity
and consequently $\phi(\omega)$ has bounded additive complexity.

Next, assume that $(i)$ does not hold. We will exhibit a finite $T \subset \mathbb{Z}$ and a non-erasing morphism $\phi: S^\infty \mapsto T^\infty$ such that $\phi(\omega)$ has unbounded additive complexity. We may assume that $\omega$ has bounded additive complexity, or else we can take $T = S$ and $\phi$ equal to the identity. By Theorem~\ref{thm: 2} there is $M \in\mathbb{N}$ such that $|\sum B - \sum B'| < M$ whenever $B,B' \in \mathcal{F}(\omega)$ are equally long. For all $n \in \mathbb{N}$ we define
\[ C(n) = \max \{ ||B|_s - |B'|_s| : B,B' \in \mathcal{F}^n(\omega), s \in S\}. \]
By Lemma~\ref{lemma: 7}, the sequence $\{ C(n) \}_{n\in \mathbb{N}}$ is unbounded, and so we may choose a subsequence $\{ C(n_k) \}_{k\in \mathbb{N}}$ such that $C(n_k) \to \infty$ as $k\to\infty$. By the pigeonhole principle there is $s \in S$ and a strictly increasing sequence $\{k_j\}_{j\in \mathbb{N}}$ such that, for all $j \in \mathbb{N}$, there are $B'_j,B^{\prime \prime}_j \in \mathcal{F}^{n_{k_j}}(\omega)$ with 
\[ |B'_j|_s - |B^{\prime \prime}_j|_s = C\left(n_{k_j} \right). \]
If $s+1\in S$ we take $T=S\backslash\{s\}$, otherwise $T = (S \backslash \{s\}) \cup \{s + 1\}$. Let $\phi: S^\infty \mapsto T^\infty$ be generated by the assignments $t \mapsto t$ for all $t \in S \backslash \{s\}$ and $s \mapsto s+1$. Then
\[ \sum \phi(B'_j) - \sum \phi(B^{\prime \prime}_j) = C\left(n_{k_j} \right) + \left( \sum B'_j - \sum B^{\prime \prime}_j) \right) \geq C\left( n_{k_j} \right) - M \to \infty \]
as $j \to \infty$. By Theorem~\ref{thm: 2}, $\phi(\omega)$ does not have bounded additive complexity. 
\end{proof}

\subsection{Anchored Words and Splicing} \label{sub: 3.3}

A word $\omega \in S^\mathbb{N}$ is called \emph{anchored} if there exists a finite alphabet $T$, a non-erasing morphism $\phi: T^\infty \mapsto S^\infty$, and a word $\omega' \in T^\mathbb{N}$ such that $\phi(\omega') = \omega$. If there exists $n$ such that $\omega \setminus [1,n]$ is anchored, then $\omega$ is \emph{eventually anchored}. Notice that eventually anchored words have rational slope. 

It follows from Theorem~\ref{thm: 7} that every uniformly recurrent word with bounded additive complexity and rational slope is eventually anchored. On the other hand, we also showed that not every recurrent word with bounded additive complexity and rational slope is eventually anchored. Thus, the question arises as to how one may construct all words with bounded additive complexity and rational slope. 

Our purpose here is to make a conjecture concerning the above problem, taking into account the work done in Subsection~\ref{sub: 2.4}. To state the conjecture, we need a definition.

\begin{defn} 
Let $n \in \mathbb{N}$ be given and suppose $\omega_1, \dots, \omega_n$ be infinite words, each over a finite integer alphabet. A \emph{splice} of $\{\omega_1, \dots, \omega_n\}$ is any word $\omega$ of the form 
\[ \omega = B_{1,1} B_{2,1} \cdots B_{n,1}B_{1,2}\cdots B_{n,2} \cdots, \]
where $B_{i,1}B_{i,2} \dots B_{i,j}$ is a prefix of $\omega_i$, for all $j \in \mathbb{N}$ and all $i=1,2,\dots,n$. In particular, we leave open the possibility that $B_{i,j}$ is empty for some $i,j \in \mathbb{N}$, so long as infinitely many $B_{i,j}$ are nonempty.
\end{defn}

We observe the following closure property of words with bounded additive complexity.

\begin{lemma} \label{lemma: 8}
If $\omega_1,\dots, \omega_n$ are words with bounded additive complexity, all of which have slope $\alpha$, then any splice $\omega$ of $\{\omega_1,\dots,\omega_n\}$ has bounded additive complexity and slope $\alpha$.
\end{lemma}

\begin{proof}
We can permute an arbitrary factor $B \in \mathcal{F}(\omega)$ to obtain a word $B'$ of the form $B' = A_1A_2\dots A_n$,
where $A_i \in \mathcal{F}(\omega_i)$ for $1 \leq i \leq n$. By Theorem~\ref{thm: 3}, there are positive integers $M_i$, depending only on $\omega_i$, such that 
\[ \left| \sum A_i - \alpha|A_i| \right| \leq M_i \]
for all $i = 1,2,\dots,n$. Therefore,
\[ \left| \sum B - \alpha|B| \right| = \left| \sum B' - \alpha |B'| \right| \leq \sum_{1 \leq i \leq n} \left| \sum A_i - \alpha|A_i| \right| \leq \sum_{1 \leq i \leq n} M_i, \]
and in view of Theorem~\ref{thm: 3}, dividing by $|B|$ yields the result. 
\end{proof}

We observe that the word constructed in Subsection 2.4 is a splice of a set containing three anchored words, each of which has slope $1$; these words are $\omega_1 = 111 \cdots$, $\omega_2 = 020202\cdots$ and $\omega_3 = 202020 \cdots$. Indeed, we conjecture that every word with bounded additive complexity and rational slope has this form.  

\begin{q} Prove or disprove the following conjecture: every word with bounded additive complexity and rational slope is a splice of a finite set of eventually anchored words, each of which has the same slope.
\end{q}

\section{Words With Constant Additive Complexity} \label{sec: 4}

Another noteworthy inquiry pertaining to additive complexity is that of what sequences $\{s_n\}_{n=1}^\infty$ over  $\mathbb{N}$ satisfy $\rho_{\omega}^{\Sigma}(n) = s_n$ for all $n \in \mathbb{N}$ and some infinite word $\omega$ over a finite integer alphabet. In particular, we are interested in finding words, if there are any, which have constant additive complexity. The question of existence, without further qualification, may be answered by considering the words $\omega_n \in \{1,2,\dots,n\}^\mathbb{N}$ given by $\omega_n(i) = i$, $i=1,2,\dots,n-1$ and $\omega_n(i) = n$, $i \geq n$. Clearly, $\omega_n$ has $\rho_{\omega_n}^{\Sigma}(i) = n$ for all $i \in \mathbb{N}$. We also note that $\rho_{\omega_n}^{\psi}(i) = n$ for all $i \in \mathbb{N}$. This was observed by Currie and Rampersad \cite{currie10}.

The question of whether or not there exist recurrent words with constant additive complexity equal to any given natural number cannot be answered so easily. Indeed, for the abelian complexity, Currie and Rampersad \cite{currie10} have shown that there are no recurrent words $\omega$ having $\rho_{\omega}^{\psi}(n) = k$ for all $n \in \mathbb{N}$, if $k \geq 4$. Our next result shows the existence of, for any given $k \in \mathbb{N}$, recurrent words $\omega$ with $\rho_{\omega}^{\Sigma}(n) = 2k + 1$ for all $n \in \mathbb{N}$.

\begin{thm} \label{thm: 11}
Let $k \in \mathbb{N}$. There exists a recurrent word $\omega \in \{0,1,2,\dots,2k\}^\mathbb{N}$ having $\rho_{\omega}^{\Sigma}(n) = 2k + 1$ for all $n \in \mathbb{N}$. 
\end{thm}

\begin{proof}
Let $k \in \mathbb{N}$ be given. For all $i \in \mathbb{N}$, let $<_i$ be a total ordering on the set of words over $\{0,1,2,\dots,k\}$ having length $i$. Thereby, define the total ordering $<$ on $\{0,1,2,\dots,k\}^*$ by setting $B' < B$ if and only if
\[ |B'| < |B| \qquad \textup{or} \qquad |B| = |B'| = i \textup{ and } B' <_i B, \qquad \forall B,B' \in \{0,1,2\dots,k\}^*. \]
We then define $\omega_<' \in \{0,1,2\dots,k\}^\mathbb{N}$ as the word obtained by concatenating all elements of $\{0,1,2\dots,k\}^*$ in increasing order under $<$. Note that $\omega_<'$ is recurrent.

Now, we let $\phi: \{0,1,2\dots,k\}^\infty \mapsto \{0,1,2,\dots,2k\}^\infty$ be the anchor generated by $\phi(i) = i(2k - i)$, $0 \leq i \leq k$. Applying it, we obtain the word $\omega_< = \phi(\omega_<')$; we note that $\omega_<$ is recurrent. By examining separately the cases $n$ odd and $n$ even, we will show that $\rho_{\omega_<}^\Sigma(n) = 2k + 1$ for all $n \in \mathbb{N}$.

First, let $n = 2j + 1$, for some $j \in \mathbb{N}$. Then for $B \in \mathcal{F}^n(\omega_<)$, there exists $B' \in \mathcal{F}^j(\omega_<')$ such that $B = s\phi(B')$ or $B = \phi(B')s$ for some $s \in \{0,1,2,\dots,2k\}$. By definition of $\omega_<'$ and $\phi$, we have
\[ \left\{ \sum B: B \in \mathcal{F}^n(\omega_<) \right\} = \{2kj + s: s = 0,1,2,\dots,2k\}. \]
Thus, we have $\rho_{\omega_<}^\Sigma(n) = 2k + 1$ in the odd case.
 
Next, let $n = 2j$, for some $j \in \mathbb{N}$. Then for $B \in \mathcal{F}^n(\omega_<)$, there exists $B' \in \mathcal{F}^j(\omega_<')$ such that $B = \phi(B')$, or there exists $B'' \in \mathcal{F}^{j-1}(\omega_<')$ such that $B = s\phi(B'')t$ for some $s,t \in \{0,1,2,\dots,2k\}$. In the first case, $\sum B = 2kj$; in the second case, the set of possible sums for $B$ is $\{2kj + s: s = -k,\dots,-1,0,1,\dots, k\}$.
Finally, $\rho_{\omega_<}^\Sigma(n) = 2k + 1$ in the even case as well. This completes the proof.
\end{proof} 

\section{Lattice Complexities} \label{sec: 5}

For the remainder of this note, $\mu$ will denote a morphism of $S^*$ into the additive monoid $\mathbb{Z}^t$, $t \geq 1$, i.e., for all $B,B' \in S^*$, we have $\mu(BB') = \mu(B) + \mu(B')$.

\begin{defn} If $\omega \in S^\mathbb{N}$, the \emph{$\mu$-complexity} of $\omega$ is the map $\rho_{\omega}^{\mu}: \mathbb{N} \mapsto \mathbb{N}$ given by 
\[ \rho_{\omega}^{\mu}(n) = | \{ \mu(B): B \in \mathcal{F}^n(\omega)\} |. \]
Any such map is known as a \emph{lattice complexity} for $\omega$. 
\end{defn}

Thus, the additive complexity $\rho_{\omega}^{\Sigma}$ of a word $\omega \in S^\mathbb{N}$ is the $\sum$-complexity, where $\sum: S^* \mapsto \mathbb{Z}$ is generated by $\rho_{\omega}^{\Sigma}(s) = s$, $s \in S$. 
As a second example of a lattice complexity, consider an ordered alphabet $S = \{s_1,\dots,s_k\}$. Letting $\psi: S^* \mapsto \mathbb{Z}^{k}$ be the \emph{Parikh map} \cite{parikh}, i.e., the morphism generated by
\begin{align*}  \psi(s_1) &= (1,0,0,\cdots,0)  \\
\psi(s_2) &= (0,1,0,\cdots, 0) \\
\vdots \quad &\: \:\vdots \qquad \quad \vdots \\
\psi(s_k) &= (0,0,0, \cdots, 1),
\end{align*}
we see that $\rho_{\omega}^{\psi}$ is precisely the abelian complexity of $\omega$. 

Many of the results we proved for words with bounded additive complexity have analogues in this more general setting. For example, one may easily prove the following result by passing to components and applying Theorem~\ref{thm: 2}.

\begin{thm} \label{thm: 12}
A word $\omega \in S^\mathbb{N}$ has bounded $\mu$-complexity if and only if there exists $M \in \mathbb{N}$ such that for all $n \in \mathbb{N}$ and all $B,B' \in \mathcal{F}^n(\omega)$, we have $|| \mu(B) - \mu(B') || \leq M$,
where $|| \cdot ||$ is the Euclidean norm on $\mathbb{Z}^t$.
\end{thm}

We observe that Lemma~\ref{lemma: 7} is a corollary of Theorem~\ref{thm: 12}. For another example of how results on additive complexity readily generalize to results on lattice complexities, define a \emph{$k$-power modulo $\mu$} as a word $B = B_1B_2\cdots B_k \in S^*$ such that $\mu(B_1) = \mu(B_2) = \cdots = \mu(B_k)$ and $|B_1| = |B_2| = \cdots = |B_k|$. Then in view of Theorem~\ref{thm: 12}, the following is essentially due to Ardal et al.~\cite{ardal12}.

\begin{thm} \label{thm: 13}
If $\omega \in S^\mathbb{N}$ has bounded $\mu$-complexity, then $\omega$ contains a $k$-power modulo $\mu$ for every $k \in \mathbb{N}$. 
\end{thm}

For $B \in S^*$, define its \emph{slope modulo $\mu$} as $\modslope{\mu}{B} = |B|^{-1}\mu(B)$, and define the slope of $\omega \in S^\mathbb{N}$ modulo $\mu$ in the natural way. By substituting $\modslope{\mu}{}$ for $\slope{}$ and $\mu(B)$ for $\sum B$, and by replacing the assumption $\slope{\omega} \in \mathbb{Q}$ by $\modslope{\mu}{\omega} \in \mathbb{Q}^t$,
one may also prove results (we omit the details) analogous to Theorems~\ref{thm: 3},~\ref{thm: 4},~\ref{thm: 5},~\ref{thm: 6},~\ref{thm: 7},~\ref{thm: 8}, and~\ref{thm: 9}, as well as to Lemmas~\ref{lemma: 1},~\ref{lemma: 2},~\ref{lemma: 3},~\ref{lemma: 4},~\ref{lemma: 5},~\ref{lemma: 6}, and~\ref{lemma: 8}. In particular, we note that these results hold for the abelian complexity; for this special case, Theorem~\ref{thm: 9} is due to Cassaigne et al.~\cite{cassaigne12} and Theorem~\ref{thm: 3} is due to Adamczewski~\cite{adam}.

It would be interesting to study the converse question as to whether or not lattice complexities can be used to gain further insight into additive complexity.

\section{Acknowledgments}

The author is pleased to thank Veso Jungi\'{c} and Tom Brown for helpful conversations and for editorial suggestions, which have significantly improved the quality of this paper.

%
%

\end{document}